\documentclass[11pt]{amsart}

\usepackage{hyperref}
\usepackage{microtype}
\usepackage{amsmath}
\usepackage{amsthm}
\usepackage{amssymb}
\usepackage{amsfonts}
\usepackage{xspace}
\usepackage{graphicx}
\usepackage{framed}
\usepackage{float}
\usepackage[]{algorithm2e}

\newcommand{\YES} {{\sc Yes}\xspace}
\newcommand{\NO}  {{\sc No}\xspace}
\newcommand{\su}{ {\rm SU(2)}\xspace }
\newcommand{\OO}{ \mathcal{O} }
\newcommand{\unknot}{{\sc UnKnot Recogntion}\xspace }

\newcommand{\NP}{{\sf NP}\xspace}
\newcommand{\coNP}{{\sf co{\text -}NP}\xspace}

 \newtheorem{theorem}{Theorem}
 \newtheorem{lemma}[theorem]{Lemma}
 \newtheorem{observation}[theorem]{Observation}
 \newtheorem{proposition}[theorem]{Proposition}
\newtheorem{definition}[theorem]{Definition}
\newtheorem{remark}[theorem]{Remark}

\author{Syed M. Meesum}
\address{Institute of Computer Science\\{University of Wroc{\l}aw, Poland}}
\email{meesum.syed@gmail.com}

\author{T. V. H. Prathamesh}
\address{Department of Computer Science, University of Innsbruck, Innsbruck Austria}
\email{prathamesh.t@gmail.com}


\title{Unknot Recognition Through Quantifier Elimination}

\begin{document}

\maketitle

\begin{abstract}
Unknot recognition is one of the fundamental questions in low dimensional topology. 
In this work, we show that this problem can be encoded as a validity problem in the existential fragment of the first-order theory of real closed fields. This encoding is derived using a well-known result on \su representations of knot groups by Kronheimer-Mrowka. We further show that applying existential quantifier elimination to the encoding enables an \unknot algorithm with a complexity of the order $2^{\OO(n)}$, where $n$ is the number of crossings in the given knot diagram. Our algorithm is simple to describe and has the same runtime as the currently best known unknot recognition algorithms. 
\end{abstract}

\section{Introduction}
In mathematics, a knot refers to an entangled loop. The fundamental problem in  the study of knots is the question of knot recognition: can two given knots be transformed to each other without involving any cutting and pasting?
This problem was shown to be decidable by Haken in 1962 \cite{haken1961theorie} using the theory of normal surfaces.
We study the special case in which we ask if it is possible to untangle a given knot to an unknot.
The \unknot recognition algorithm takes a knot presentation as an input and answers \YES if and only if the given knot can be untangled to an unknot.
The best known complexity of \unknot recognition is $2^{\OO(n)}$, where $n$ is the number of crossings in a knot diagram \cite{burton2012fast, haken1961theorie}.

More recent developments show that the \unknot recognition is in  $\NP \cap \coNP$. 
Using the theory of normal surfaces, Hass, Lagarias and Pippenger \cite{hass1999computational} proved existence of an \NP membership certificate for \unknot.
A notion which is closer to the commonly accepted notion of untangling a knot is that of using Reidemeister moves.
The existence of a polynomial length sequence of Reidemeister moves having size $\OO(n^{11})$, that untangles an unknot, was proved to exist by Lackenby \cite{lackenby2015polynomial}. Searching over all possible Reidemeister moves will give a simple to describe algorithm having runtime of $2^{\OO(n^{11})}$. 
According to Lackenby \cite{lackenby2015polynomial}, a proof sketch for \coNP membership of \unknot was first announced by Agol, but not written down in detail.
Assuming the Generalized Reimann Hypothesis, a polynomial-time certificate for non-membership of a knot in \unknot was proved to exist by Kuperberg \cite{kuperberg2014knottedness}.
Finally, an unconditional proof for the membership of \unknot in \coNP was given by Lackenby \cite{lackenby2016efficient}.

Several approaches to unknot recognition can be found in literature, like complete knot invariant such as Khovanov homology, branching algorithms \cite{burton2012fast} involving normal surface theory, manifold hierarchies\cite{lackenby2016efficient}, Dynnikov diagrams \cite{dynnikov2006arc}, and  equational reasoning \cite{fish2016efficient}. 

Most of the known algorithms deciding \unknot are complex and have an involved analysis. 
The authors believe that this paper presents a simpler alternate algorithm, which relies on reducing the above problem to a sentence in the existential theory of reals~\cite{Renegar89}. This enables application of the decision procedure for existential theory of reals using quantifier elimination, to obtain an algorithm which is exponential in complexity, thus of the same complexity class as the best known approaches.\\

\textbf{Acknoweldgements:} The authors would like to thank the Institute of Mathematical Sciences, HBNI, Chennai, India, where a part of the work was carried out. 
The first author is supported by the NCN grant number 2015/18/E/ST6/00456. The second author is supported by the FWF project number P30301.
\section{Preliminaries}


This section contains some of the basic definitions in knot theory, and a brief note on quantifier elimination and existential theory of reals without explicitly stating the algorithm. 
For a more detailed introduction to knot theory one may refer to \cite{lackenby2016elementary, crowell2012introduction, lickorish2012introduction, kawauchi2012survey}, and for quantifier elimination in existential theory of reals, one may refer to \cite{basu2007algorithms}.

For a natural number $n$, we use $[n]$ to denote the set $\{1, 2,\dots, n\}$.
We use $\su$ to denote the group which contains $2\times 2$ complex hermitian matrices with unit determinant, with multiplication as the group operation. For a natural number $d$, we use $S^d$ to denote the subset of $\mathbb{R}^d$ with euclidean norm equal to one. The symbol $i$ denotes $\sqrt{-1}$, the imaginary root of unity. The symbol $\wedge$ is used to denote the operation of logical conjunction. The symbol $\vee$ is used to denote the operation of logical disjunction.

\subsection{Knot Theory}

\subsubsection{Basic Definitions}

\begin{definition}
A (tame) knot $K$ is the image of a smooth injective map from $S^1$ to $S^3$.
\end{definition}

\begin{remark} $S^3$ in the above definition can be replaced by $\mathbb{R}^3$. But we use $S^3$, because some of the concepts that we introduce such as knot groups, exist only in the context of the embedding of a circle in $S^3$.
\end{remark}

Two knots are considered to be the same, if they are related by an equivalence condition called ambient isotopy. It is defined as follows:

\begin{definition}[Ambient Isotopy]
The knots $K_1$ and $K_2$ are ambient isotopic if there exists a smooth map $F:S^3 \times [0,1] \rightarrow S^3$ such that:
\begin{itemize}
\item $\forall x \in S^3.\ F(x, 0) = x$.
\item $F(K_1, 1) = K_2$.
\item $\forall t \in [0,1].\ F(S^3, t)$ is a homeomorphism of $S^3$.  
\end{itemize}
\end{definition}

Ambient isotopy describes when a knot can be transformed into another by a deformation that does not involve any cutting and pasting. To draw a knot on paper, we use the convention that wherever the string is shown broken it is assumed to be passing under the unbroken string. To illustrate the above condition, consider the following knots:

\setlength{\tabcolsep}{3.0em}
\begin{figure}[H]
\begin{center}
\begin{tabular}{c c c}
{
\includegraphics[scale=0.21]{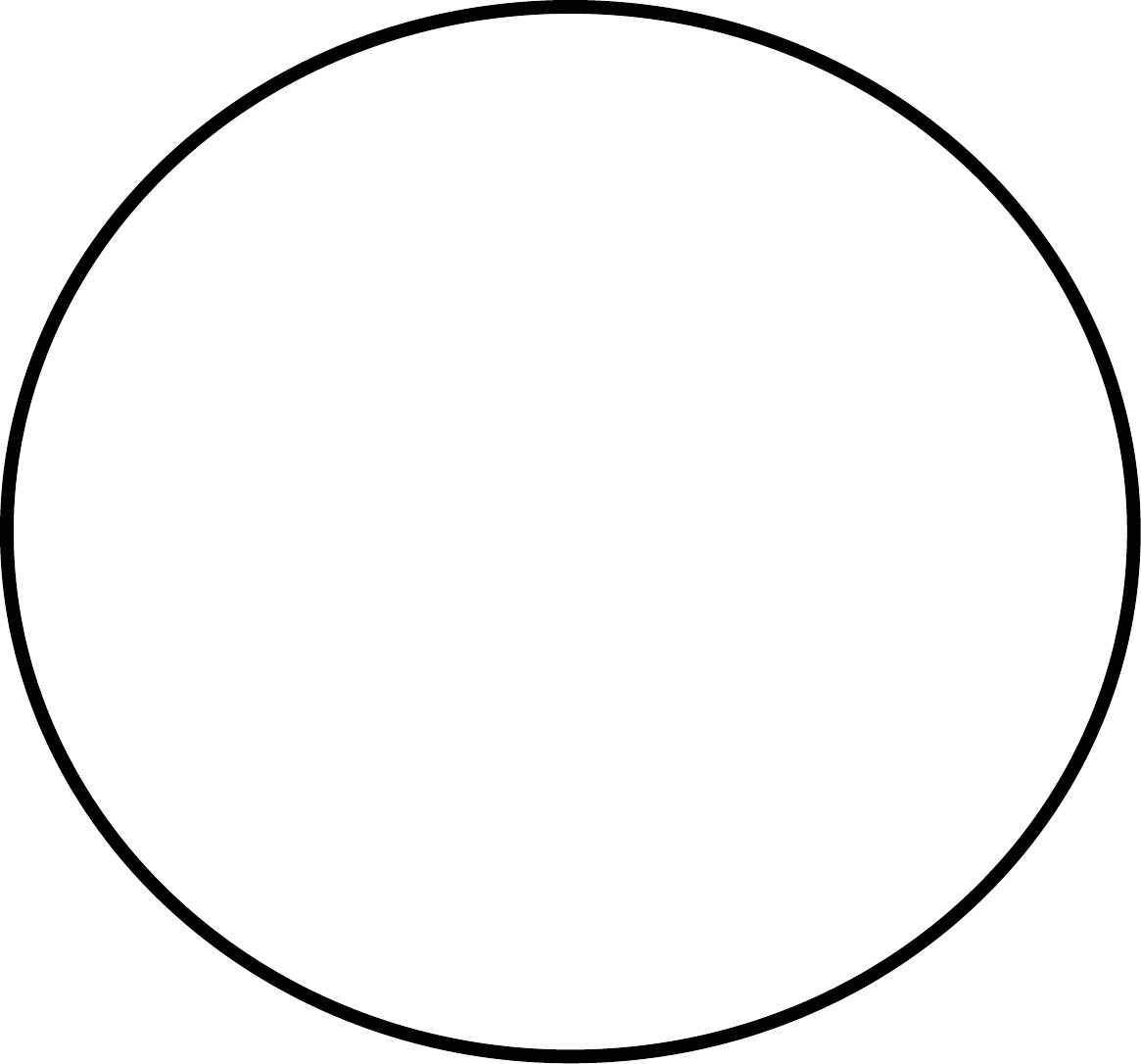}
}
& 
{
\includegraphics[scale=0.15]{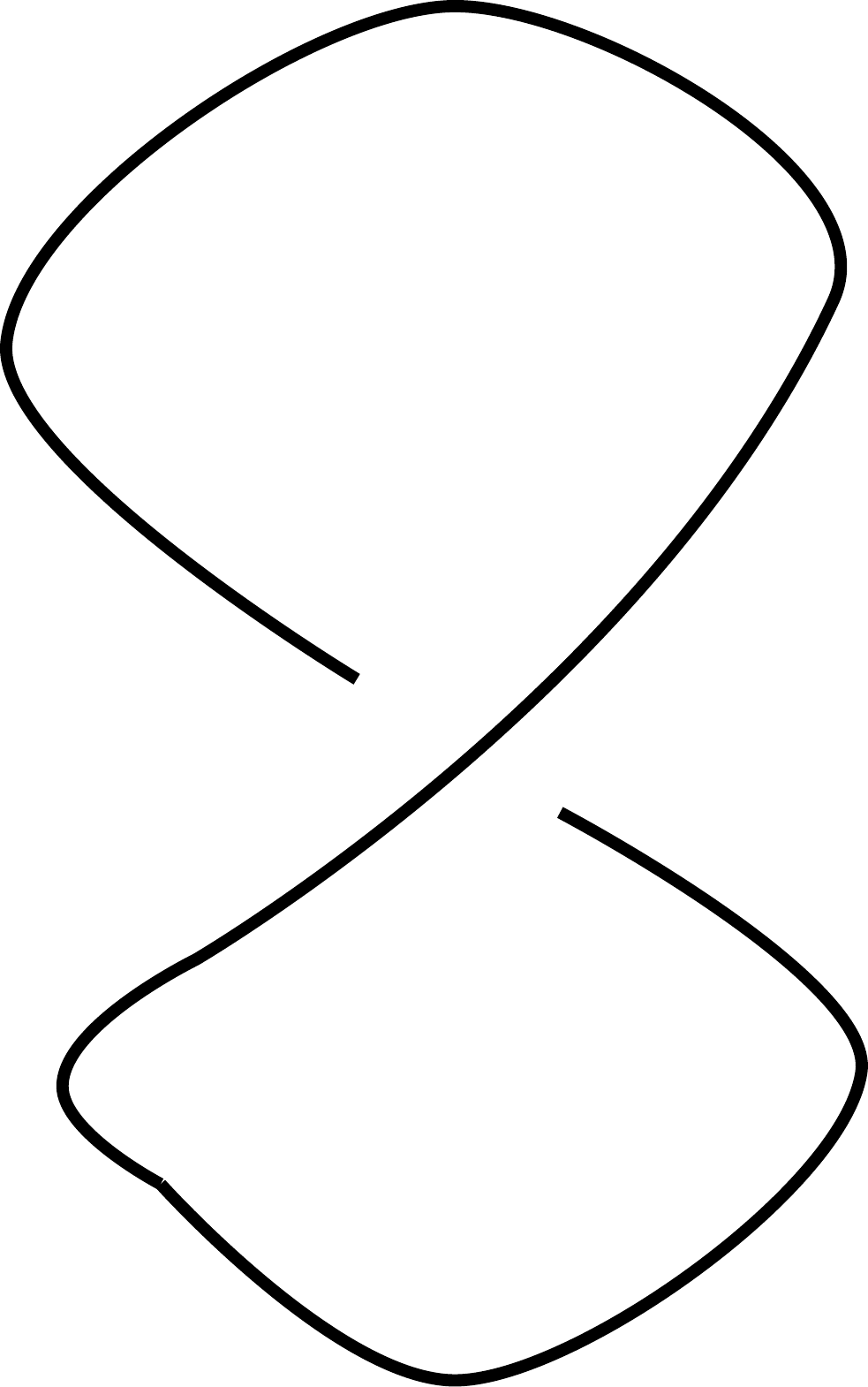}
}
&
{
\includegraphics[scale=0.12]{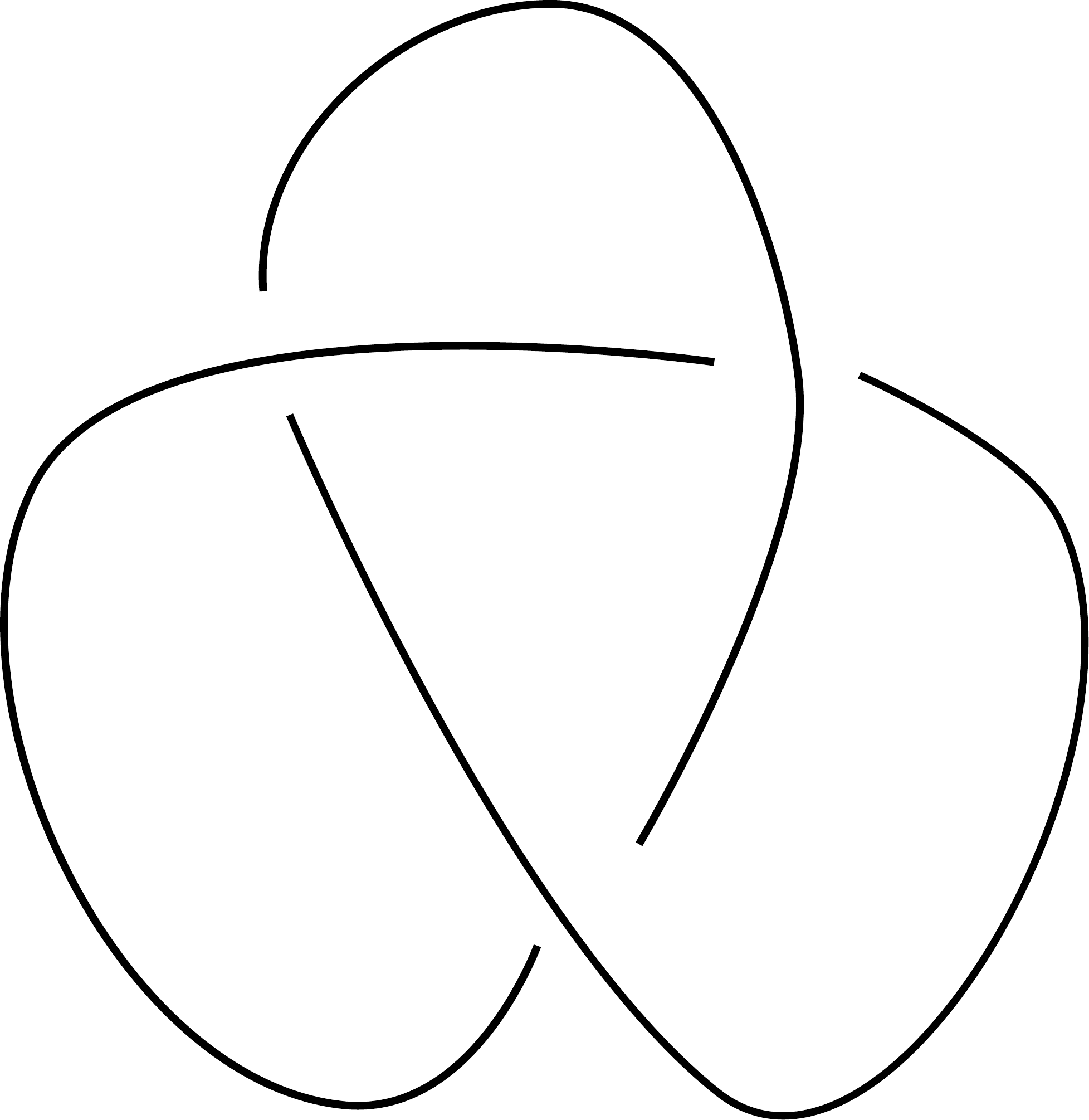}
}\\
1) Unknot & 2) A Twist & 3) Trefoil Knot\\
\end{tabular}
\end{center}
\end{figure}

The first two knot diagrams shown above can be deformed into each other by twisting/untwisting, thus they represent the same knot.  Deforming either of the first two knots into the third knot, would involve cutting and pasting. Thus it is different from the former knots. 

\begin{definition} An \emph{unknot} is a knot which is ambient isotopic to the circle $S^1$. A knot $k$ is \emph{knotted} if and only if it is not an unknot. 
\end{definition}

Determining when two diagrams represent the same knot, is the key problem of knot theory.  The special case of it, determining when a given knot diagram is equivalent to unknot is called the unknot recognition problem. There have been several algorithms and approaches to the knot recognition, listed in the previous section.

\subsubsection{Knot Group} 

\begin{figure}[b]
\centering %
\includegraphics[scale=0.6]{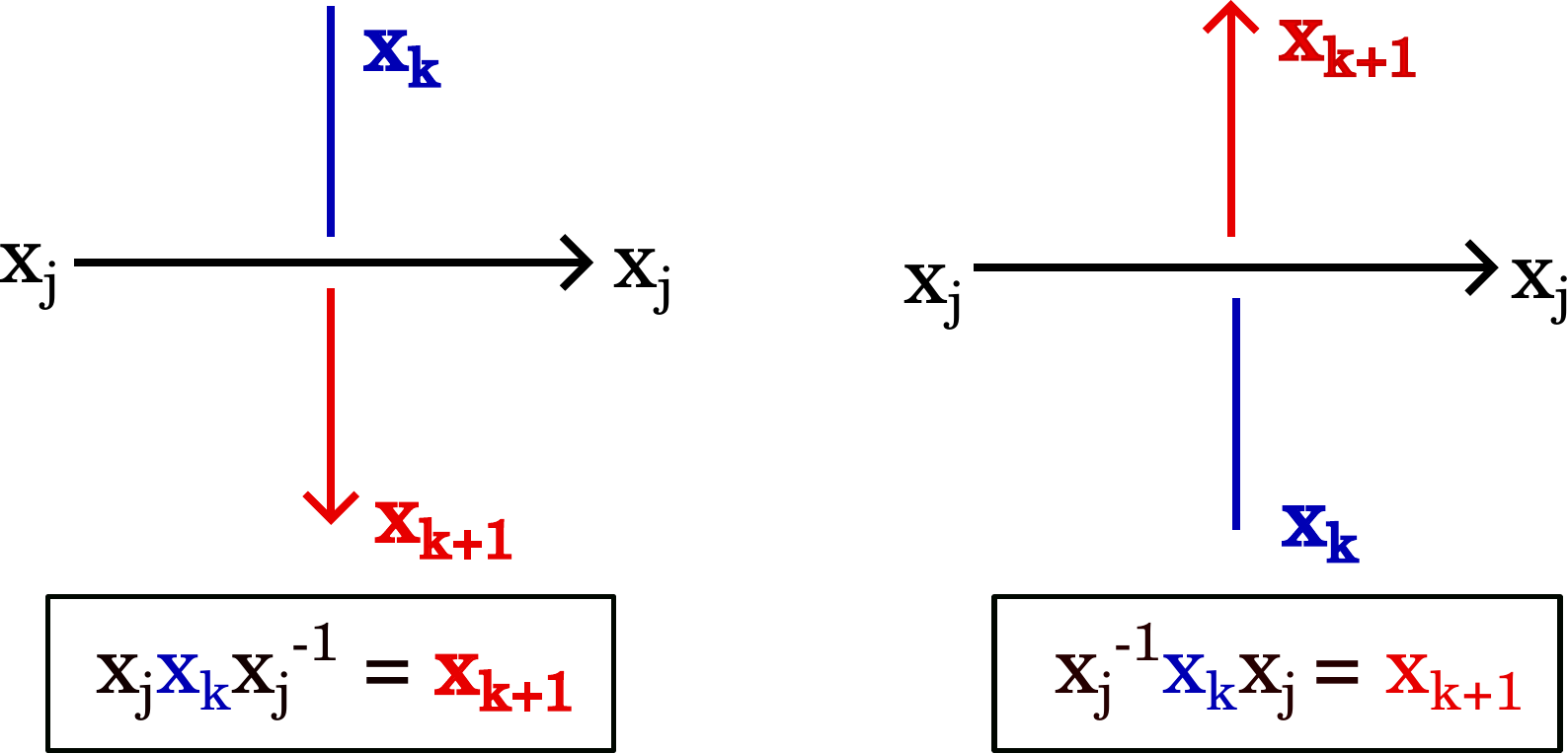}
\caption{Wirtinger presentation relations for the knot group.}
\label{fig:crossing-group}
\end{figure}

One of the known invariants of knots is the fundamental group of the knot complement. Knot complement refers to the compact 3-manifold obtained by considering the complement of a tubular neighbourhood of the knot. This invariant can detect knots up to mirror image. Presentations of this group, called the Wirtinger presentation, can be easily computed from a  knot diagram in the following manner:
\begin{itemize}
\item The knot diagram is oriented in one of the two possible directions. The string constituting the knot is given a direction which fixes the direction of all the arcs occurring in the knot diagram.
\item Every connected arc is associated to a distinct generator.
\item Every crossing gives rise to a relation in the presentation. The relation depends on the orientation of the arcs in the crossing, in the manner as shown in Figure \ref{fig:crossing-group}.
\end{itemize} 

Computing the Wirtinger presentation of a group from the diagram can be achieved using the steps described above in time which is a linear function of the number of crossings in the given knot diagram.

\subsubsection{$\su$ representations of the knot group}\label{ss:km} The following theorem by Kronheimer-Mrowka, translates unknot recoginition to existence of non-commutative $\su$ representations of the knot group.

\begin{proposition}[\cite{kronheimer2004witten}, \cite{kuperberg2014knottedness}]
\label{thm:embedding}
 If $K$ is knotted, then it has an non-commutative $\su$ representations of the knot group $\pi_1(S^3 \setminus K)$.
\end{proposition}
 The following lemma is derived from the theorem above. The reverse direction of the lemma follows from the fact the knot group of the unknot is $\mathbb{Z}$, and all its $\su$ representations are commutative. 
 \begin{lemma}
 \label{cor:embedding}
 A knot $K$ is  knotted if and only if there exists a non-commutative $\su$ representation of the knot group $\pi_1(S^3 \setminus K)$.
\end{lemma}
We note that every finitely presented group has a trivial homomorphism to the
group $\su$ via a mapping of each generator to the identity matrix.

\subsection{Quantifier Elimination in Existential Theory of Reals}

Decidability of the first-order existential theory of reals refers to the existence of a decision procedure for validity of all sentences of the following form:

\[\exists \bar{X}.\ F(\bar{X}) \]

Where $F$ is a conjunction of polynomial equalities and inequalities in real variables $\bar{X}$.  It follows from the Tarski-Seidenberg theorem that the above problem is decidable by the quantifier elimination algorithm. The quantifier elimination in this case, in fact holds true for deciding validity of all first-order sentences. Quantifier elimination algorithm refers to computation of a quantifier free sentence, which is equivalent to the sentence with quantifiers. Validity of the quantifier free sentences can be computed, which makes the algorithm a decision procedure for the first-order theory. Quantifier elimination algorithm in existential fragment is restricted to finding equivalent quantifier free sentences only for first-order sentences with existential quantifiers, of the form described above.

The known complexity bounds for quantifier elimination in the general first-order theory of reals are doubly exponential. The existential fragment has a much lower complexity bound, as stated in the following result: 
\begin{proposition}[see Proposition $4.2$ in \cite{Renegar89}]
  \label{theorem:real-feasibility}
  Given a set $\mathcal{P}$ of equations, each of which is either a $\ell$ polynomial equalities or inequality of degree $d$ in $k$ variables,
and  with integer coefficients of bit length $L$, we can decide the feasibility of $\mathcal{P}$ with 
  $L \log L \log \log L(\ell\cdot d)^{\OO(k)}$ bit operations.
  \looseness-1
\end{proposition}

\section{Algorithm}

The algorithm {\sc Unknot-QE} appears as Algorithm~\ref{algo:matrig}
on the next page.

\begin{algorithm}[]
\begin{framed}
\DontPrintSemicolon
{\noindent\bf Algorithm:} {\sc Unknot-QE}\;
\SetAlgoLined
\KwData{A knot group $\pi_1(S^3\setminus K) = \langle\ g_1, g_2, \dots ,g_n\ |\ R_1, R_2, \dots, R_n \ \rangle$}

\KwResult{Output \YES if $K$ is an unknot, otherwise output \NO}
\Begin{
  $ \mathcal{E}  \longleftarrow  \emptyset $ \;
  $ \mathcal{P}  \longleftarrow  \emptyset $ \;
  \For{$k \leftarrow 1$ \KwTo $n$ }{
    $M_k \longleftarrow 
      \begin{bmatrix}
      a_k + i b_k & c_k + i d_k \\
      -c_k+ i d_k & a_k - i b_k
      \end{bmatrix}$ \;
   }
  \For{$k \leftarrow 1$ \KwTo $n$}{
    \If{$R_k=g_j g_k g_j^{-1} g_{k+1}^{-1}$}
    {
      $E_k \longleftarrow M_{k+1} M_j - M_j M_k$\;
    }
    \If{  $ R_k=g_j^{-1} g_k g_j g_{k+1}^{-1} $  }
    {
      $E_k \longleftarrow M_k M_j - M_j M_{k+1}$\;
    }
    $E_k^{\rm Re} \longleftarrow \operatorname{Re^{U}}(E_k)$ (the real part of the entries of the first row of $E_k$)\;
    $E_k^{\rm Im} \longleftarrow \operatorname{Im^{U}}(E_k)$ (the complex part of the entries of the first row of $E_k$)\;
    
        Put all the polynomials in $E_k^{\rm Re}$ and $E_k^{\rm Im}$ in the set $\mathcal{P}$ \;
        Put $a_k^2 + b_k^2 + c_k^2 + d_k^2 - 1$ in $\mathcal{P}$\;
        }
        Put the equation $\sum_{p \in \mathcal{P}} p^2 =0$ in $\mathcal{E}$ \;
         
         $\mathcal{N} \longleftarrow \emptyset$ \;
    \For{$k \leftarrow 2$ \KwTo $n$}
    	{
    	Put $a_k-a_1$, $b_k - b_1$, $c_k - c_1$ and $d_k - d_1$ in $\mathcal{N}$ \;
	}
	Put the inequality $\sum_{p \in \mathcal{N}} p^2 \neq 0$ in $\mathcal{E}$ \;
	
	\eIf{$\mathcal{E}$ is satisfiable}
	{\Return\ \YES \;}
      	{\Return\ \NO \;}
}
\end{framed}
\caption{Description of the algorithm for Unknot.\label{algo:matrig}}
\end{algorithm}

\begin{remark} The algorithm can be simplified leading to improvements in efficiency, within the same complexity class, but our choice of description is motivated by expository concerns.
\end{remark}

The key idea behind the algorithm can be stated in terms of the following theorem which will be proved in the next section.
\begin{theorem}
\label{thm:qe_formula} 
There exists a computable map $\Phi$, which takes a knot diagram $K$ to  a sentence in the existential fragment of the first order theory of reals. $K$ is knotted if and only if $\Phi(K)$ is valid. Applying existential quantifier elimination algorithm to $\Phi(K)$ leads to a decision method for \unknot.
\end{theorem}

\section{Proof of the Algorithm}

In the proof, we reduce the Kronheimer-Mrowka property, stated in section \ref{ss:km}, to a first-order sentence in existential theory of quantifier elimination. Observe that every knot group has Wirtinger presentations which correspond to knot diagrams. These presentations are of the following form:
\[\langle\ g_1, g_2, \dots ,g_n\ |\ R_1, R_2, \dots, R_n \ \rangle.\] 
For $i \in [n]$, the symbol $g_i$ denotes a generator of the group and $R_i$
denotes a relation satisfied by the generators.
In the Wirtinger presentation, each $R_k$ is 
 either $g_jg_kg_j^{-1}g_{k+1}^{-1}$ or $g_j^{-1}g_k g_j g_{k+1}^{-1}$, where
 $j \in [n]$ and depends on $k$, 
we use $+(R_k)$ or $-(R_k)$ to denote them respectively.

Finding a non-commutative $\su$ representation of $\pi_1(S^3 \setminus K)$, if it exists, can be seen as a conjunction of the following steps:

\begin{enumerate}
 \item Mapping generators of the Wirtinger presentation to matrices in $\su$.
 \item Checking that the above map extends to a well defined homomorphism, i.e. the matrices corresponding to generators satisfy the generating relations of the Wirtinger presentation.
 \item Checking that the map is non-commutative. 
\end{enumerate}

In the following paragraphs, we elaborate on and construct equivalent conditions for each of the above steps. Let $g_k$ be a generator in the Wirtinger presentation, associated to a knot diagram. Consider a map $\Phi$ from the set of generators to $\mathcal{M}$, in which we map $g_k$ to $M_k$. 

\begin{align}
M_k = \begin{bmatrix}
\label{eqn:mat}
a_k + i b_k & c_k + i d_k \\
-c_k+ i d_k & a_k - i b_k \end{bmatrix}
\end{align}
Where $a_k$, $b_k$, $c_k$, $d_k$ are real variables.  For $M_k$ to be an element of $\su$, it must be unitary (i.e. the inverse of $M_k$ is equal to transpose of its complex-conjugate) and it must have unit determinant, which gives us the following extra condition on the variables used to define it.

\begin{observation}(folklore)
\label{obs:1}
$M_k \in \su$ if and only if $(a_k^2 + b_k^2 + c_k^2 + d_k^2 = 1)$.
\end{observation}

 In addition, the mapped elements $M_k$'s  have to satisfy the  knot group relations obtained from the Wirtinger presentation i.e.
 \begin{equation}
 \label{eq:1}
 (+(R_k) \rightarrow M_j  M_k M_j^{-1}  M_{k+1}^{-1}  = I) \wedge (-(R_k) \rightarrow M_j^{-1}  M_k  M_j M_{k+1}^{-1} = I)
 \end{equation} 
 where $I$ is the $2\times 2$ identity matrix.

For $k \in [n]$, we define $E_k$ as follows:

\[E_k =
\begin{cases} 
     M_{k+1} M_j - M_j M_k  & +(R_k)\\
   M_k M_j - M_j M_{k+1} & -(R_k)
   \end{cases}
\]

The condition on matrices in Equation \ref{eq:1} can be restated in terms ot $E_k$ as follows:
\begin{observation}
\label{obs:2}
For $M_k \in \su$, for $i \in [n]$, a knot group embedding must satisfy the following,
\[
E_k = \begin{bmatrix}
0 & 0 \\
0 & 0 \end{bmatrix}
\]
\end{observation}

The above observation meets the goal of step (2). The above matrix equality can be rewritten as a system of four quadratic equations in real variables in the following fashion:
\begin{itemize}
\item Decompose the matrix $E_k$ into real and imaginary parts -- $Re(E_k)$ and $Im(E_k)$: then $E_k = 0$ if and only if $Re(E_k) = 0$ and $Im(E_k) = 0$. 
\item Define $Re^U(E_k)$ and $Im^U(E_k)$ to be the sets of polynomial equalities:
\[p(x) = 0\]
Where $p(x)$ is an entry in the top row of the $Re(E_k)$ and $Im(E_k)$ respectively. We similarly define  $Re^D(E_k)$ and $Im^D(E_k)$ for the bottom row. Either by simplifying $E_k$ or by noticing the fact that the matrices $M_k$ form a group and their product matrix must also be of the same form as Equation~\ref{eqn:mat}, one can observe that:
\[Re^U(E_k) \cup Im^U(E_k) = Re^D(E_k) \cup Im^U(E_k)\]
\end{itemize}

Consider the set $\mathcal{P}$, consisting of all the polynomials $Re^U(E_k)$, $Im^U(E_k)$ and $a_k^2 + b_k^2 + c_k^2 +d_k^2 - 1 = 0$, where $k \in [n]$. The following lemma allows us to decrease the number of equalities we have in the system of equations.

\begin{lemma}[Reverse Rabinoswitch Encoding \cite{passmore2009combined}]
\label{lem:eq-and}
Let $\mathcal{P}=\{p_1= 0, \dots,p_m =  0\}$ be the system of $m$ equality constraints, as defined above. Then $p_1 = 0 \wedge p_2 = 0 \dots \wedge p_m =  0$ is satisfiable if and only if $\sum_{i\in [m]} p_i^2 = 0$ is satisfiable.
\end{lemma}

The above equation gives an equivalent condition for checking the existence of a $\su$ representation of a knot group. We need to further ensure that the representation is non-commutative. In general, to check that the generators are non-commutative, we would have to check that at least one of the pairs of generators does not commute. However, the special structure of knot group relations allows for a much simpler encoding into polynomial inequalities. In the following lemma we show that finding a non-commutative $\su$ representation is equivalent to finding a representation which maps at least two distinct generators of the Wirtinger presentation to distinct elements of $\su$.  

\begin{lemma} 
\label{lem:non-commutative}
 A knot group $\pi_1(S^3\setminus K)$, with generators $g_i$, has a non-commutative homomorphism $\rho$
 to a group if and only if $\rho(g_i) \neq \rho(g_j)$, for some
 $i\neq j$.
\end{lemma}

\begin{proof}
In the forward direction, observe that if the generator's images are all equal then $\rho$ is commutative.
In the backward direction, assume that the image set of $\{g_i\}_{1 \leq i \leq n}$ has at least two distinct elements. Therefore, there must exist an index $k\in [n]$ such that $\rho(g_k) \neq \rho(g_{k+1})$.
Without loss of generality assume that the relation $R_k=+(R_k)$, similar steps would be true for the $-(R_k)$ form of the relations. Since $\rho(R_k) = I$, we have
\[\rho(g_{j}) \rho(g_k) \rho(g_{j})^{-1} \rho(g_{k+1})^{-1} = I .\]
As $ \rho(g_k) \neq \rho(g_{k+1} )$, it must be the case that  
\begin{align*}
&\rho(g_{j}) \rho(g_k) \rho(g_{j})^{-1} \neq \rho(g_k) \\
\implies &\rho(g_{j}) \rho(g_k) \neq \rho(g_k) \rho(g_{j}).
\end{align*}
Therefore $\rho$ is non-commutative. 
\end{proof}

If $\rho$ is the $\su$ representation, then it suffices to check that there exist at least two distinct matrices in the 
image to obtain the existence of a non-commutative representation, in addition to the earlier mentioned constraints. The following series of observations further simplify the criterion:

\begin{observation}
\label{obs:11}
Consider the matrices $M_j$ and $M_k$, as defined above where $j,k \in [n]$.  
\[(M_j \neq M_k) \leftrightarrow (a_j \neq a_k\ \vee\ b_j \neq b_k\ \vee\ c_j \neq c_k\ \vee\ d_j \neq d_k)\]
\end{observation}

\begin{observation}
\label{obs:12}
Let $r_1, \dots, r_m$ be real numbers. There exist indices $j,k \in [n]$ such that $r_j \neq r_k$ if and only if $\bigvee_{\ell=2}^m (r_1 \neq r_{\ell})$ is true.
\end{observation}

The following lemma allows us to convert the system of inequalities encoding the constraint of non-commutativity into just one equivalent inequality.

\begin{lemma}
\label{lem:neq-or}
Let $\mathcal{N}=\{p_1\neq 0, \dots,p_m\neq 0\}$ be a system of $m$ inequality constraints. Then $p_1\neq 0 \vee p_2 \neq 0 \dots \vee p_m\neq 0$ is satisfiable if and only if $\sum_{i\in [m]} p_i^2 \neq 0$ is satisfiable.
\end{lemma}
\begin{proof}
The lemma follows from the negation of the statement of Lemma~\ref{lem:eq-and}.
\end{proof}

Combining Lemmas~\ref{lem:non-commutative},~\ref{lem:neq-or} and Observations~\ref{obs:11},~\ref{obs:12}, we get that it suffices to add the the following inequality, to check non-commutativity: 
\[ {\sum}_{i=1}^n  (a_i - a_1)^2 + (b_i - b_1)^2+ (c_i - c_1)^2 + (d_i - d_1)^2 \neq 0\]

Let $\mathcal{E}$ be the set consisting of above inequality and the equation in Lemma \ref{lem:eq-and}. It is easy to see from the Lemma \ref{lem:eq-and}, Observations \ref{obs:11} and \ref{obs:12}  and the Lemma \ref{lem:neq-or}, that there exists a non-commutative representation from the knot group to $\su$, if and only if $\mathcal{E}$ has a solution. This completes the proof of Theorem~\ref{thm:qe_formula}.

\section{Complexity Analysis}

The algorithm consists of first computing Wirtinger presentation from the input knot diagram, which can be done in linear time. The formula $\mathcal{E}$ can be constructed in polynomial time.
Next, we analyse the complexity of deciding the feasibility of the constructed existential formula.
If the number of crossings in the provided knot diagram is $n$ then the number of real variables in the system of equations is $4n$. The system of equations $\mathcal{E}$ consists of exactly two statements; one equality and one inequality, maximum total degree of any monomial in it is $4$. Finally, note that the coefficients of polynomials occurring in our system of equations is from the set $\{-2-1,1,2\}$, as the coefficients of the polynomials  before squaring are from the set $\{-1,1\}$.  Using Proposition~\ref{theorem:real-feasibility}, we get the following result.

\begin{theorem}
\label{theorem:algo-runtime}
The procedure {\sc Unknot-QE} solves the problem \unknot in time $2^{\OO(n)}$, where $n$ is the number of crossings in the given knot diagram.
\end{theorem}

\section{Conclusion}
In this article, we presented an algorithm for \unknot, a proof of correctness, and an analysis of its complexity. The key advantage of this algorithm over the existent algorithms is the simplicity of description while having the same runtime complexity as the other currently best algorithms. As an open problem, is it possible to reduce the runtime complexity further ? It may be possible to do so by decreasing the number of variables in the equation via some substitution methods.

\bibliographystyle{plainurl}
\bibliography{ref.bib} 

\end{document}